\documentclass[12pt]{amsart}
\usepackage[english]{babel}
\usepackage{amsthm}
\usepackage{inputenc}
\usepackage{amssymb}
\usepackage{color}
\usepackage{cite}

\newtheorem{thm}{Theorem}[section]
\newtheorem{cor}[thm]{Corollary}
\newtheorem{lm}[thm]{Lemma}
\newtheorem{prop}[thm]{Proposition}
\newtheorem{defn}[thm]{Definition}
\newtheorem{rem}[thm]{Remark}
\newtheorem{exam}[thm]{Example}
\numberwithin{equation}{section}

\begin{document}

\author[Ayupov]{Shavkat Ayupov}
\email{sh$_{-}$ayupov@mail.ru}
\address{ V.I.Romanovskiy Institute of Mathematics, Uzbekistan Academy of Sciences, 29, Dormon Yoli street, 100125  Tashkent,   Uzbekistan}

\author[Kudaybergenov]{Karimbergen Kudaybergenov}
\email{karim2006@mail.ru}
\address{Department of Mathematics, Karakalpak State University, Ch. Abdirov 1, Nukus 230113, Uzbekistan}

\author[Omirov]{Bakhrom Omirov}
\email{omirovb@mail.ru}
\address{ National University of Uzbekistan,  4, University street, 100174, Tashkent,   Uzbekistan}

\author[Zhao]{Kaiming Zhao}
\email{kzhao@math.ac.cn}
\address{Department of Mathematics, Wilfrid Laurier University, Waterloo, Ontario, N2L
3C5, Canada}
\address{College of Mathematics and Information Science, Hebei Normal (Teachers)
University, Shijiazhuang 050016, Hebei, PR China}

\title[Semisimple Leibniz algebras]{Semisimple Leibniz algebras and their derivations and automorphisms}

\maketitle

\begin{abstract}
The present paper is devoted to the description of finite-dimensional semisimple Leibniz algebras over complex numbers, their derivations and automorphisms.
\end{abstract}

\medskip

\medskip \textbf{AMS Subject Classifications (2010):
17A32, 17A60, 17B10, 17B20.}

\textbf{Key words:} Lie algebra, Leibniz algebra, simple algebra, semisimple algebra, irreducible module,
derivation, automorphism.

\section{Introduction}

In recent years the non-associative analogues of classical constructions become of interest in connection with their applications in many branches of mathematics and physics. Leibniz algebras present a ``non-commutative" analogue of Lie algebras and they were introduced by Loday \cite{Lod} as algebras satisfying the (right) Leibniz identity:
$$[x,[y, z]]=[[x, y], z] - [[x, z], y].$$ Leibniz algebras preserve an important property of Lie algebras: the operator of right multiplication is a derivation.

During the last decades the theory of Leibniz algebras has been actively investigated. Some (co)gomology and deformation  properties; results on various types of decompositions; structure of solvable and nilpotent Leibniz algebras;
 classifications of some classes of graded nilpotent Leibniz algebras were obtained in  numerous papers devoted to Leibniz algebras, see, for example,  \cite{AlbAyupov1,Bal, p-fil, quasi-fil, Lod-Pir, AlbAyupov2} and
 reference therein. In fact, many results on Lie algebras have been extended to the Leibniz algebra case. For instance, the classical results on Cartan subalgebras
 \cite{Cartan}, Levi's decomposition \cite{Barnes}, properties of solvable algebras
 with given nilradical \cite{Nilradical} and others from the theory of Lie algebras are also true for Leibniz algebras.

From the classical theory of finite-dimensional Lie algebras it is known that an arbitrary semisimple Lie algebra is decomposed into a direct sum of simple ideals, which are completely classified \cite{Jac}. In the paper \cite{example} an example of semisimple Leibniz algebra, which can not be  decomposed into a direct sum of simple ideals, is presented (see Example 3.6 also). This shows that the structure of semisimple   Leibniz algebras is much more complicated than structure of semisimple Lie algebras. Thus, the natural problem arises - to describe semisimple Leibniz algebras. In fact, the structure depends on relations between semisimple Lie algebras and their modules. Due to Barnes' result \cite{Barnes} an arbitrary semisimple Leibniz algebra $\mathcal{L}$ is represented as $\mathcal{L}=\mathcal{S}\dot{+}\mathcal{I}$, where $\mathcal{S}$ is a semisimple
 Lie algebra and $\mathcal{I}$ is the ideal generated by squares of elements of the algebra $\mathcal{L}$.
This means that the problem is focused to investigation of the relation between the ideal $\mathcal{I}$ and the semisimple Lie algebra $\mathcal{S}$.

One of the main results of the present  paper is the description on finite-dimensional semisimple Leibniz algebras (Theorem 3.5). Then next steps, which are crucial for any variety of algebras, are investigations of derivations and automorphisms of these algebras. Note that derivations and automorphisms of semisimple Lie algebras are completely described by their actions on corresponding simple ideals \cite{Jac}. The descriptions of derivations and automorphisms of simple Leibniz algebras were studied in \cite{AKO}. Here we present the description on derivations of semisimple Leibniz algebras. Since any semisimple Leibniz algebra can be represented as a direct sum of indecomposable semisimple Leibniz algebras, the Lie algebra  of derivations of such algebras is also represented as a direct sum of the Lie algebras of all derivations on indecomposable semisimple Leibniz algebras.
The automorphism group of semisimple Leibniz algebras is quite different from that of simple Leibniz algebras, and more subtle than their derivations.

The paper is organized  as follows.

In Section 2 we give some preliminaries from the Leibniz algebra
theory.

Section~3 is devoted to the description of the structure of finite-dimensional complex
semisimple Leibniz algebras (Theorem 3.5). We also determined a structure  of irreducible modules
over a Lie algebra which is a direct sum of two Lie
algebras (see Theorem~\ref{twotensor}).

In Section~4  we completely determine  derivations of finite-dimensional complex semisimple Leibniz algebras (Theorem 4.5).

Finally, In Section~5  we determine the group of  automorphisms of finite-dimensional complex semisimple Leibniz algebras (Theorem 5.7).

Throughout the paper, vector spaces and algebras are finite-dimensional over the field of complex numbers $\mathbb{C}$. We will use $\oplus$ to denote direct sum of ideals in an algebra, while use $\dot +$ to denote direct sum of subalgebras or subspaces. Moreover, in
the table of multiplication of an algebra the omitted products  in terms of a given basis are assumed to be zero.

\section{Preliminaries}

In this section we give necessary definitions and results on semisimple Lie algebras and their irreducible modules towards the description of semisimple Leibniz algebras.

\begin{defn} An algebra $(\mathcal{L},[\cdot,\cdot])$ over a field $\mathbb{F}$ is called a
Leibniz algebra if it satisfies the property
\begin{center}
\([x,[y,z]]=[[x,y],z] - [[x,z],y]\) for all \(x,y \in
\mathcal{L},\)
\end{center}
which is called Leibniz identity.
\end{defn}

The Leibniz identity is a generalization of the Jacobi identity since under the
condition of anti-symmetricity of the product ``[$\cdot\,
,\,\cdot$]'' this identity changes to the Jacobi identity. In
fact, the definition above is the notion of the right Leibniz
algebra, where ``right'' indicates that any right multiplication
operator is a derivation of the algebra. In the present  paper the
term ``Leibniz algebra'' will always mean the ``right Leibniz
algebra''. The left Leibniz algebra is characterized by the
property that any left multiplication operator is a derivation.

Below we present an adapted version to our further using  of the well-known Schur's Lemma (\cite[P.
57, Corollary 3]{Zhel}).

\begin{lm}  \label{schur} Let $\mathcal{G}$ be a complex Lie algebra,  let $\mathcal{V}=\mathcal{V}_1\oplus\cdots
\oplus\mathcal{V}_m$ and $\mathcal{W}=\mathcal{W}_1\oplus\cdots
\oplus\mathcal{W}_n$ be a completely reducible $\mathcal{G}$-modules, where \(\mathcal{V}_1, \ldots,
\mathcal{V}_n, \mathcal{W}_1,\cdots
,\mathcal{W}_n\) are irreducible modules. Then any $\mathcal{G}$-module homomorphism $\varphi : \mathcal{V}
\rightarrow \mathcal{W}$ can be represented as
\[
\varphi=\sum\limits_{i=1}^m\sum\limits_{j=1}^n \lambda_{i j}\varphi_{i j},
\]
where the operator \(\varphi_{i j}:
\mathcal{V}_i\to\mathcal{W}_j
\)  are fixed module homomorphisms  and  \(\lambda_{i j}\) are
 complex numbers. Furthermore   \(\varphi_{i j}\ne0\)  if and only if  \(\varphi_{i j}\) is an  isomorphism if and only if \(\mathcal{V}_i\) and \(\mathcal{W}_j\)  are isomorphic $\mathcal{G}$-modules. \end{lm}

\begin{defn} A Lie algebra $\mathcal{G}$ is called to be  reductive, if $\mathcal{R(G)}=Z(\mathcal{G}),$ where  $\mathcal{R(G)}$ and $Z(\mathcal{G})$ are the solvable radical and the center of $\mathcal{G}$, respectively.

\end{defn}

\begin{lm}\label{Prop19.1} \cite[Proposition 19.1]{Hum}
(a) Let $\mathcal{G}$ be a reductive Lie algebra. Then
$\mathcal{G}=[\mathcal{G},\mathcal{G}]\oplus Z(\mathcal{G})$, and
$[\mathcal{G},\mathcal{G}]$ is either semisimple or \(0.\)

b) Let $\mathcal{V}$ be a finite-dimensional space and let
$\mathcal{G}\subset \mathfrak{gl}(\mathcal{V})$  be a nonzero Lie
algebra acting irreducibly  on $\mathcal{V}.$ Then $\mathcal{G}$
is reductive with $\dim Z(\mathcal{G})\leq 1$. If in addition
$\mathcal{G}\subset \mathfrak{sl}(\mathcal{V})$ then $\mathcal{G}$
is semisimple.
\end{lm}

\begin{thm}\label{Cor 21.2} \cite[Corollary 21.2]{Hum}.
Let $\mathcal{G}$ be a semisimple Lie algebra. Then the map
$\lambda \mapsto \mathcal{V}_\lambda$ induces a one-one
correspondence between the set of all dominant integral functions
$\Lambda^+$ and the isomorphism classes of finite-dimensional
irreducible $\mathcal{G}$-modules.
\end{thm}

We need the following result from \cite[Page 143, Theorem
6]{ZheWan}.

\begin{thm}\label{Zhe6}  Let \(\mathcal{G}=\mathcal{G}_1\oplus\mathcal{G}_2\)  be a semisimple Lie algebra. Let a Cartan
subalgebra \(\mathcal{H}\)  of \(\mathcal{G}\)  be correspondingly
decomposed as \(\mathcal{H}=\mathcal{H}_1\oplus\mathcal{H}_2,\)
where \(\mathcal{H}_i\)  is a Cartan subalgebra of
\(\mathcal{G}_i\)  (\(i = 1, 2\)). Suppose that the fundamental root system
 \(\Pi\) of \(\mathcal{G}\)  with respect to
\(\mathcal{H}\)  is decomposed into \(\Pi=\Pi_1\cup \Pi_2,\) where
\(\Pi_1=\left\{\alpha_1, \ldots, \alpha_n\right\}\) and
\(\Pi_2\left\{\beta_1, \ldots, \beta_m\right\}\)  is the
fundamental root system   of \(\mathcal{G}_i\)  with respect to
\(\mathcal{H}_i\)  (\(1 = 1, 2\)).
\begin{itemize}
\item[(i)] Suppose \(\varrho_i\)  is an irreducible representation of \(\mathcal{G}_i\)  with
highest weight \(\omega_i\)  and representation space
\(\mathcal{V}_i\)  (\(i = 1, 2\)). If
\[
\varrho(x)(v_1\otimes v_2)=\varrho_1(x_1)v_1\otimes v_2+v_1\otimes
\varrho_2(x_2)v_2,
\]
\[
x=x_1+x_2,\,\, x_i\in \mathcal{G}_i,\,\, i = 1, 2,
\]
then \(\varrho\)  is an irreducible representation of
\(\mathcal{G}\)  with highest weight \(\omega_1+\omega_2,\) where
\[
(\omega_1+\omega_2)(h)=\omega_1(h_1)+\omega_2(h_2),\,\,
h=h_1+h_2,\,\, h_i\in\mathcal{H}_i,\,i=1, 2.
\]
\item[(ii)] Conversely, every irreducible representation of \(\mathcal{G}\)
is obtained as in (i).
\end{itemize}
\end{thm}

For a Leibniz algebra $\mathcal{L}$, a subspace generated by its squares $\mathcal{I}=\text{span}\left\{[x,x]:  x\in \mathcal{L}\right\}$
due to Leibniz identity becomes an ideal, and the quotient $\mathcal{G}_\mathcal{L}=\mathcal{L}/\mathcal{I}$ is a Lie algebra
called liezation of $\mathcal{L}.$ Moreover,  $[ \mathcal{L}, \mathcal{I}]=0.$ In general, $[\mathcal{I}, \mathcal{L}]\ne 0$.
Since we are interested in Leibniz algebras which are not Lie algebras, we will always assume that $\mathcal{I}\ne0$.

\begin{defn}
A Leibniz algebra $\mathcal{L}$ is called  simple  if  its liezation is a simple Lie algebra and the ideal $\mathcal{I}$ is a
simple ideal. Equivalently, $\mathcal{L}$ is simple iff $\mathcal{I}$ is the only non-trivial ideal of $\mathcal{L}.$
\end{defn}

\begin{defn}\label{construction} Let  $\mathcal{G}$ be a Lie algebra  and  $\mathcal{V}$ a  (right) $\mathcal{G}$-module. Endow the vector space $\mathcal{L}=\mathcal{G} \oplus \mathcal{V}$ with the brackets as follows:
$$
[(g_1, v_1), (g_2, v_2)]:=([g_1, g_2], v_1.g_2),
$$
where  $v. g$ (sometimes denoted as $[v, g]$)  is an action of an element $g$ of $\mathcal{G}$ on
$v\in \mathcal{V}.$ Then $\mathcal{L}$ is a Leibniz algebra, denoted as $\mathcal{G}\ltimes V$.
\end{defn}

\begin{defn}
A Leibniz algebra $\mathcal{L}$ is called \textit{semisimple} if
its liezation $\mathcal{G}$ is a semisimple Lie algebra.
\end{defn}

The following theorem recently proved by D.~Barnes \cite{Barnes}
presents an analogue of Levi--Malcev's theorem for Leibniz
algebras.

\begin{thm} \label{t2}  \label{thmBarnes} Let $\mathcal{L}$ be a finite dimensional Leibniz algebra over a
field of characteristic zero and let $\mathcal{R}$ be its solvable
radical. Then there exists a semisimple Lie subalgebra
$\mathcal{S}$ of $\mathcal{L}$ such that
$\mathcal{L}=\mathcal{S}\dot{+}\mathcal{R}.$
\end{thm}

Application of the analogue of Levi's Theorem for Leibniz algebras to the case of semisimple Leibniz algebras implies the following

\begin{prop}
Let $\mathcal{L}$ be a finite-dimensional semisimple  Leibniz
algebra. Then $\mathcal{L}= \mathcal{S}\oplus \mathcal{I},$ where
$\mathcal{S}$ is a semisimple Lie subalgebra of $\mathcal{L}$.  Moreover  $[\mathcal{I}, \mathcal{S}]=\mathcal{I}$.
\end{prop}

It should be noted that $\mathcal{I}$ is a nontrivial module over the Lie algebra $\mathcal{S}$.

The structure of simple  Leibniz algebras is clear. By Proposition~\ref{construction} a simple  Leibniz algebra is a sum of a simple Lie algebra $\mathcal{S}$
and a simple module (the highest weight module) over $\mathcal{S}$. Our problem concerns semisimple Leibniz algebras.

Semisimple Leibniz algebras in general, are not direct sum of simple Leibniz algebras.
See the next example.

\begin{exam}
Let $\mathcal{L}$ be a $10$-dimensional Leibniz
algebra. Let $\{e_1, h_1, f_1, e_2, h_2, f_2, \\ x_1, x_2, x_3,
x_4\}$ be a basis of $\mathcal{L}$ such that
$\mathcal{I}=\{x_1, x_2, x_3, x_4\},$ and multiplication table of
$\mathcal{L}$ has the following form:
$$
[\mathfrak{sl}_2^i, \mathfrak{sl}_2^i]: \quad
\begin{array}{lll}
\ [e_i,h_i]=2e_i, & [f_i,h_i]=-2f_i, & [e_i,f_i]=h_i, \\
\ [h_i,e_i]=-2e_i & [h_i,f_i]=2f_i,  & [f_i,e_i]=-h_i, \ i=1, 2, \\
\end{array}
$$
$$
[\mathcal{I},\mathfrak{sl}_2^1]: \quad
\begin{array}{llll}
\, [x_1,f_1]=x_2, & [x_1,h_1]=x_1, & [x_2,e_1]=-x_1, & [x_2,h_1]=-x_2, \\
\, [x_3,f_1]=x_4, & [x_3,h_1]=x_3, & [x_4,e_1]=-x_3, & [x_4,h_1]=-x_4, \\
\end{array}
$$
$$
[\mathcal{I},\mathfrak{sl}_2^2]: \quad
\begin{array}{llll}
\, [x_1,f_2]=x_3, & [x_1,h_2]=x_1, & [x_3,e_2]=-x_1, & [x_3,h_2]=-x_3, \\
\, [x_2,f_2]=x_4, & [x_2,h_2]=x_2, & [x_4,e_2]=-x_2, & [x_4,h_2]=-x_4, \\
\end{array}
$$
(omitted products of the basis vectors are equal to zero).
\end{exam}

From this table of multiplications we conclude that $\mathcal{L}$ is semisimple and 
$[\mathcal{I},\mathfrak{sl}_2^1] =[\mathcal{I},\mathfrak{sl}_2^2]
=\mathcal{I}.$ Moreover, $\mathcal{I}$ splits over
$\mathfrak{sl}_2^1$ (i.e. $\mathcal{I}=\text{span}\{x_1,x_2\}\oplus
\text{span}\{x_3,x_4\}$) and over $\mathfrak{sl}_2^2$ (i.e.
$\mathcal{I}=\text{span}\{x_1,x_3\}\oplus\text{span} \{x_2,x_4\}$). Actually, $\mathcal{I}$ is a simple ideal (an irreducible $(\mathfrak{sl}_2^1\oplus
\mathfrak{sl}^2_2)$-module). Therefore,
$
\mathcal{L}=(\mathfrak{sl}_2^1\oplus
\mathfrak{sl}^2_2)\ltimes \mathcal{I} $ cannot be a direct sum of two nonzero ideals. See Example 3.6 for a more general case.

\section{Structure of semisimple Leibniz algebras}


First we give the generalized version of Theorem~\ref{Zhe6}, which has an independent interest.

\begin{thm}\label{twotensor}
Let $\mathcal{G}=\mathcal{G}_1\oplus \mathcal{G}_2$ be a direct
sum of Lie algebras and let $\mathcal{V}$ be a finite-dimensional
irreducible $\mathcal{G}$-module   with $\dim \mathcal{V}
>1.$ Then there are  finite-dimensional irreducible $ \mathcal{G}_i$-modules $\mathcal{V}_i$ so that
$
\mathcal{V} \cong
\mathcal{V}_{ 1}\otimes \mathcal{V}_{ 2}.
$
\end{thm}

\begin{proof}
Consider   $\textrm{Ann}(\mathcal{V})=\left\{x\in \mathcal{G}:
 \mathcal{V}.x =0\right\}$  which is an ideal of
$\mathcal{G}$. Then $\mathcal{V}$ is a faithful irreducible module
over $\overline{\mathcal{G}}=\mathcal{G}/
\textrm{Ann}(\mathcal{V}),$ in particular, for an element
\(x\in \overline{\mathcal{G}}\) from \(v.x=0\) for all \(v\in
\mathcal{V}\) it follows that \(x=0.\)

Let $\rho$ be the representation of $ \overline{\mathcal{G}}$ on $\mathcal{V}$, i.e.,
$\rho(\bar g)(v)=v.\bar g$ for any $v\in \mathcal{V}$ and $\bar g\in \mathcal{G}$. Then we have the injective linear map $\rho:  \overline{\mathcal{G}}\to \mathfrak{gl}(\mathcal{V})$. Thus
$\overline{\mathcal{G}}$ is also finite-dimensional.

Applying part b) of  Lemma~\ref{Prop19.1} to   $\overline{\mathcal{G}}$  we know that $\bar{\mathcal{G}}$ is
reductive,  $Z(\bar{\mathcal{G}})\leq 1$, and  $\bar{\mathcal{G}}=[\bar{\mathcal{G}}, \bar{\mathcal{G}}]\oplus Z(\bar{\mathcal{G}})$.
Denote  the natural projections by
$\rho_1:\bar{\mathcal{G}}\to [\bar{\mathcal{G}}, \bar{\mathcal{G}}] $ and
$\rho_2:\bar{\mathcal{G}}\to Z(\bar{\mathcal{G}})$, and define
$\pi=\rho_1\circ\rho$.

Assume $Z(\overline{\mathcal{G}})=\mathbb{C}z\ne0$. Since \([[v, x],
z]=[[v, z], x]\) for all \(x\in \overline{\mathcal{G}}\) and
\(v\in \mathcal{V},\) it follows that the mapping
\[
v\in \mathcal{V} \hookrightarrow [v, z] \in \mathcal{V}
\]
is a $\mathfrak{\overline{\mathcal{G}}}$-module homomorphism of
$\mathcal{V}.$  Taking into account that  $\mathcal{V}$ is an
irreducible $\mathfrak{\overline{\mathcal{G}}}$-module, by Lemma \ref{schur}  there exists \(\lambda \in \mathbb{C}\) such that
\([v, z]=\lambda v\) for all \(v\in \mathcal{V}.\)

We know that    $\mathcal{V}$ is a simple module over $\pi({\mathcal{G}})=[\bar{\mathcal{G}}, \bar{\mathcal{G}}]$.
If $\pi(\mathcal{G}_1) \cap \pi(\mathcal{G}_2)\ne\{0\}$, it has to be semisimple. We see that
$$\pi(\mathcal{G}_1) \cap\pi( \mathcal{G}_2)=[\pi(\mathcal{G}_1) \cap\pi( \mathcal{G}_2),\pi(\mathcal{G}_1 )\cap\pi( \mathcal{G}_2)\}]\subset [\pi(\mathcal{G}_1),\pi( \mathcal{G}_2)\}]=0,$$ a contradiction. Thus  $\pi(\mathcal{G}_1) \cap \pi(\mathcal{G}_2)=\overline{0},$ and
$\pi({\mathcal{G}})=\pi(\mathcal{G}_1) \oplus
\pi(\mathcal{G}_2)$.

Since
$\mathcal{V}$ is also an irreducible module over
$ \pi(\mathcal{G}_1)\oplus  \pi(\mathcal{G}_2)$, then both $ \pi(\mathcal{G}_1)$ and $  \pi(\mathcal{G}_2)$ are semisimple. By Theorem~\ref{Cor 21.2}, we deduce that  there are  finite-dimensional irreducible $\pi( \mathcal{G}_i)$-modules $\mathcal{V}_i$ so that
$
\mathcal{V} \cong
\mathcal{V}_{ 1}\otimes \mathcal{V}_{ 2}$. We can easily extend $\pi( \mathcal{G}_i)$-module $\mathcal{V}_i$
into a $ \mathcal{G}_i$-module. This completes the proof.
\end{proof}

 We say
that a semisimple Leibniz algebra
\(\mathcal{L}=\mathcal{S}\dot{+}\mathcal{I}\) is
\textit{decomposable}, if
\(\mathcal{L}=\left(\mathcal{S}_1\dot{+}\mathcal{I}_1\right)\oplus
\left(\mathcal{S}_2\dot{+}\mathcal{I}_2\right),\) where $\mathcal{S}_1\dot{+}\mathcal{I}_1$ and
$\mathcal{S}_2\dot{+}\mathcal{I}_2$ are non-trivial semisimple Leibniz
algebras. Otherwise, we say that \(\mathcal{L}\) is
\textit{indecomposable.}

\begin{lm}\label{decompos} Any   semisimple Leibniz algebra has the form:
\[
\mathcal{L}=\bigoplus\limits_{i=1}^n \left(\mathcal{S}_i\dot{+}
\mathcal{I}_i\right),
\]
where \(\mathcal{S}_i\dot{+}\mathcal{I}_i\) is an indecomposable
semisimple Leibniz algebra for all \(i\in \{1, \dots, n\}.\)
\end{lm}

\begin{proof}
Let \(\mathcal{L}=\mathcal{S}\dot{+} \mathcal{I}\) be a semisimple
Leibniz algebra and let \(\mathcal{S}=\bigoplus\limits_{i=1}^n
\mathcal{S}_i\) be a decomposition   of simple Lie
ideals $\mathcal{S}_i$. We know that $[\mathcal{I}, \mathcal{S}]=\mathcal{I}$.  The proof is by induction on \(n.\)

Let  \(n=1,\) that is, \(\mathcal{S}\) is a simple Lie algebra.
Then \(\mathcal{L}=\mathcal{S}\dot{+} \mathcal{I}\) is an
indecomposable semisimple Leibniz algebra.

Suppose that the assertion of the theorem is true for all numbers
least than \(n\) and \(\mathcal{S}=\bigoplus\limits_{i=1}^n
\mathcal{S}_i.\)

Consider a partition of the set \(\{1, 2, \dots, n\}=A\cup B\) into union of
disjoint subsets. Set
\begin{center}
\(\mathcal{I}_A=\left[\mathcal{I}, \bigoplus\limits_{i\in A}
\mathcal{S}_i\right]\) and \(\mathcal{I}_B=\left[\mathcal{I},
\bigoplus\limits_{j\in B} \mathcal{S}_j\right]. \)
\end{center}

Case 1. Let \(\mathcal{I}_A\cap \mathcal{I}_B\neq \{0\}\) for any
non trivial partition \(A\cup B\) of the set \(\{1, 2, \dots, n\}.\) In this
case \(\mathcal{L}\) is indecomposable.

Case 2. Let \(\mathcal{I}_A\cap \mathcal{I}_B=\{0\}\) for some  partition \(A\cup B\) of the set \(\{1, 2, \dots, n\}.\) Then
\(\mathcal{L}\) is decomposable and
\[
\mathcal{L}=\left(\bigoplus\limits_{i\in A}
\mathcal{S}_i\dot{+}\mathcal{I}_A\right)\oplus
\left(\bigoplus\limits_{j\in B}
\mathcal{S}_j\dot{+}\mathcal{I}_B\right).
\]
By the hypothesis of the induction algebras
\(\bigoplus\limits_{i\in A} \mathcal{S}_i\dot{+}\mathcal{I}_A\)
and \(\bigoplus\limits_{j\in B}
\mathcal{S}_j\dot{+}\mathcal{I}_B\) can be represented as a direct
sum of indecomposable algebras.  The proof is complete.
\end{proof}

Further we need the following auxiliary result.

\begin{lm}\label{lm00}
Let \(\mathcal{S}\) be a semisimple Lie algebra and let
\(\mathcal{I}\) be an irreducible \(\mathcal{S}\)-module. Then
\([\mathcal{I}, \mathcal{S}]=0\) if and only if \(\dim
\mathcal{I}=1.\)
\end{lm}

\begin{proof} Let $\rho$ be the representation of ${\mathcal{S}}$ on $\mathcal{I}$, i.e.,
$\rho( g)(v)=[v,g]$ for any $v\in \mathcal{I}$ and $ g\in \mathcal{S}$. Then we have the Lie algebra homomorphism  $\rho:   {\mathcal{S}}\to \mathfrak{gl}(\mathcal{I})$. So $\rho( {\mathcal{S}})$ is a subalgebra of $ \mathfrak{gl}(\mathcal{I})$. The result follows from Lemma 2.4.
\end{proof}

\begin{rem}\label{rem23} From Lemma~\ref{lm00} we conclude that for an indecomposable semisimple Leibniz algebra \(\mathcal{L}=\mathcal{S}\dot{+}\mathcal{I}\) the inequality \(\dim \mathcal{I}\geq 2\) holds true if $\mathcal{I}\ne0$.
\end{rem}

From Lemma \ref{decompos}, we need only to determine the structure
of indecomposable semisimple Leibniz algebras. Suppose
\(\mathcal{L}\) is  an indecomposable semisimple Leibniz algebra
with a given Levi decomposition
$\mathcal{L}=\mathcal{S}\dot{+}\mathcal{I}.$ We may assume that \(
\mathcal{S}=\oplus _{i=1}^m \mathcal{S}_i\) and
\(\mathcal{I}=\oplus _{i=1}^n \mathcal{I}_i\)  where each
\(\mathcal{S}_i\) is a simple Lie algebra and each
\(\mathcal{I}_i\) is an irreducible $S$-module. We say that
\(\mathcal{S}_i\)  and  \(\mathcal{S}_j\) are {\it  adjacent}  if
there exists  \(\mathcal{I}_k\) such that $ [\mathcal{I}_k,
\mathcal{S}_i]= [\mathcal{I}_k, \mathcal{S}_j]=\mathcal{I}_k$. We
say that  \(\mathcal{S}_i\)  and  \(\mathcal{S}_j\) are  {\it
connected} if there exist  \(\mathcal{S}_{k_1}=\mathcal{S}_i,\
\mathcal{S}_{k_2}, \cdots, \mathcal{S}_{k_r}=\mathcal{S}_j\) such
that  \(\mathcal{S}_{k_l}\) and \(\mathcal{S}_{k_{l+1}}\) are
adjacent.

Now we can prove our main result on indecomposable semisimple Leibniz algebras.

\begin{thm} \label{thm25}
Let \(\mathcal{L}=\mathcal{S}\dot{+}\mathcal{I}\)   be an
indecomposable semisimple Leibniz algebra with
$\mathcal{I}\ne\{0\}$. Then
\begin{itemize}
\item[(a).]   \( \mathcal{S}=\oplus _{i=1}^m \mathcal{S}_i\) where each \(\mathcal{S}_i\) is a simple Lie algebra;

\item[(b).]  \(\mathcal{I}=\oplus _{i=1}^n \mathcal{I}_i\)  where each
\(\mathcal{I}_i\) is an irreducible $\mathcal{S}$-module with
$[\mathcal{I}_i,\mathcal{S} ]=\mathcal{I}_i$;
\item[(c).]  \(\mathcal{L}=\mathcal{S}\ltimes \mathcal{I}\);
\item[(d).] for any $1\le i\le m$ and $1\le j\le n$ there
is an irreducible $\mathcal{S}_i$-module $\mathcal{I}_{ij}$ such
that $\mathcal{I}_i=\otimes_{j=1}^m\mathcal{I}_{ij}$;
\item[(e).] any $\mathcal{S}_i$ and $\mathcal{S}_j$ are connected.
\end{itemize}
\end{thm}

\begin{proof} Parts (a) and (b) are clear,
(c) follows from the definition, (d) follows from
Theorem~\ref{Zhe6}.

(e) The proof is by induction on \(n.\)

Let \(n=1.\) Then \(\mathcal{I}\) is an irreducible
\(\mathcal{S}\)-module. By Theorem~\ref{Zhe6} for each $j\in \{1,
\ldots, m\}$ there is an irreducible $\mathcal{S}_j$-module
$\mathcal{J}_{j}$ such that
$\mathcal{I}=\otimes_{j=1}^m\mathcal{J}_{j}.$ Since
\(\mathcal{S}\dot + \mathcal{I}\) is indecomposable,
Lemma~\ref{lm00} implies that \([\mathcal{J}_{j},
\mathcal{S}_j]=\mathcal{J}_{j}\) for all \(j.\)  Thus
\(\left[\mathcal{I}, \mathcal{S}_j\right]=\mathcal{I}\) for all
\(j.\) This means that \(\mathcal{S}_i\)  and  \(\mathcal{S}_j\)
are adjacent.

Suppose that the assertion (e)  is true for all numbers least than
\(n>1.\) Consider a Leibniz algebra \(\mathcal{S}\dot +
\mathcal{J},\) where \(\mathcal{J}=\oplus _{i=1}^{n-1}
\mathcal{I}_i.\) Assume that \(\mathcal{S}\dot + \mathcal{J}\) is
indecomposable. By the hypothesis of the induction,
\(\mathcal{S}_i\) and \(\mathcal{S}_j\) are  already connected.

Now assume that \(\mathcal{S}\dot + \mathcal{J}\) is decomposable.
Let us consider the decomposition of \(\mathcal{S}\dot + \mathcal{J}\) into a direct sum of indecomposable semisimple
Leibniz algebras: \(\bigoplus\limits_{t=1}^p\left(\oplus_{i\in
A_t}\mathcal{S}_i \dot + \oplus _{s\in B_t}\mathcal{I}_s\right),\)
where \(A_1,\ldots ,  A_p\) and \(B_1,\ldots,  B_p\) are
partitions of \(\{1, \ldots, m\}\) and \(\{1, \ldots n-1\},\)
respectively. Since \(\mathcal{L}=\mathcal{S}\dot{+}\mathcal{I}\)
is indecomposable, for every \(s\in \{1, \ldots, p\}\) there
exists an index \(i_s\in A_s\) such that \([\mathcal{I}_n,
\mathcal{S}_{i_s}]=\mathcal{I}_n.\) Thus \(\mathcal{S}_{i_1},
\ldots, \mathcal{S}_{i_p}\) are mutually adjacent. On the other
hand, since \(\oplus_{i\in A_t}\mathcal{S}_i \dot + \oplus _{s\in
B_t}\mathcal{I}_s\) is indecomposable, by the hypothesis of the
induction, for each \(i, j \in A_s\, (1\leq s \leq p)\) algebras
\(\mathcal{S}_i\) and \(\mathcal{S}_j\) are connected. Thus
\(\mathcal{S}_i\) and \(\mathcal{S}_j\) are connected for every
\(i, j\in \{1, \ldots, m\}.\) The proof is complete.
\end{proof}

In our notations we present an example that generalizes the  one   in \cite[Theorem 4.2]{example} as a
semisimple Leibniz algebra which can not be decomposed into the direct sum
of simple ideals.

\begin{exam} \label{ex01} Let $m,n\in\mathbb{N}$.
Let  $\mathcal{L}=\left(\mathfrak{sl}_2^1\bigoplus
\mathfrak{sl}_2^2\right)\dot{+}\mathcal{I}$ be  a semisimple Leibniz algebra with a basis
$$\{e_1, h_1, f_1,e_2, h_2, f_2, x_i\otimes y_j : i=0, 1, \cdots, m; j=0, 1, \cdots, n \}$$ and the product table:
$$  \begin{array}{ll}
[e_i,h_i]=-[h_i,e_i]=2e_i,  \\[1mm]
[h_i,f_i]=-[f_i,h_i]=2f_i, \\[1mm]
[e_i,f_i]=-[f_i,e_i]=h_i,\\[1mm]
[x_i\otimes y_j,h_1]=(m-2i)x_i\otimes y_j, \\[1mm]
[x_i\otimes y_j,h_2]=(n-2j)x_i\otimes y_j,  \\[1mm]
[x_i\otimes y_j,e_1]=ix_{i+1}\otimes y_j, \\[1mm]
[x_i\otimes y_j,e_2]=jx_i\otimes y_{j+1}, \\[1mm]
[x_i\otimes y_j,f_1]=(m-i)x_{i-1}\otimes y_j,\\[1mm]
[x_i\otimes y_j,f_2]=(n-j)x_i\otimes y_j, \\[1mm]
\end{array}
$$
where $0\leq i\leq n$ and  $0\leq j\leq m$. If $V_1$ is the
$(m+1)$-dimensional irreducible $\mathfrak{sl}_2^1$-module and
$V_2$ is the $(n+1)$-dimensional irreducible
$\mathfrak{sl}_2^2$-module, then
$\mathcal{L}=\left(\mathfrak{sl}_2^1\bigoplus
\mathfrak{sl}_2^2\right)\ltimes (V_1\otimes V_2)$.
\end{exam}

\section{Derivations on semisimple Leibniz algebras}

In this section we describe all derivations of semisimple Leibniz
algebras. We will use  ${\rm{Der}}(\mathcal{L})$ to denote the Lie
algebra of derivations of a  Leibniz algebra $\mathcal{L}.$

\begin{lm}\label{lmZ1}
Let $\mathcal{L}=\mathcal{L}_1\oplus \mathcal{L}_2$ be a Leibniz
algebra such that $[\mathcal{L}_i, \mathcal{L}_i]=\mathcal{L}_i.$
Then
${\rm{Der}}(\mathcal{L})={\rm{Der}}(\mathcal{L}_1)\oplus{\rm{Der}}(\mathcal{L}_2).$
\end{lm}

\begin{proof}
This is similar to that of the Lie algebra case. We omit the details.
\end{proof}

From this lemma, in order to determine ${\rm{Der}}(\mathcal{L})$
for  semisimple  Leibniz algebras $\mathcal{L}$ we need only to
consider indecomposable semisimple  Leibniz algebras
$\mathcal{L}.$ From now on in this section we assume that
$\mathcal{L}=\mathcal{S}\dot + \mathcal{I}$ is an indecomposable
semisimple Leibniz algebra. We assume that
\begin{equation}\label{decss}
\mathcal{S}=\bigoplus\limits_{i=1}^rn_i \mathcal{S}_i
\end{equation}
is a
decomposition into the sum of  simple Lie ideals where  $S_i\not\simeq S_j$ if $i\ne j$; and that
\begin{equation}\label{decii}
\mathcal{I}=\bigoplus\limits_{i=1}^{r+s}m_i \mathcal{I}_i
\end{equation} is a
decomposition into the sum of  simple $\mathcal{S}$-modules where
$\mathcal{I}_i\not\simeq \mathcal{I}_j$ if $i\ne j.$  We may
further assume that $\mathcal{S}_i\simeq \mathcal{I}_i$ as
$\mathcal{S}$-modules and \(m_i\geq 0\) for
$i\in\{1,2,\cdots,r\}$.

 For any $x\in \mathcal{S}$ we have the inner derivation
 $$
 R_x: \mathcal{L}\to \mathcal{L}, \,\, y\mapsto [y,x], \forall \ y\in
 \mathcal{L}.
 $$
 Let $R_S=\{R_x:x\in \mathcal{S}\}$ be the inner derivations of $\mathcal{L}$, and
 $$
 \aligned
 {\rm{Der}}(L)_{\mathcal{S}, \mathcal{I}}
 &=\{d\in  {\rm{Der}}(\mathcal{L}):d(\mathcal{S})\subset \mathcal{I}, d(\mathcal{I})=0\},\\
 {\rm{Der}}(\mathcal{L})_{\mathcal{I}, \mathcal{I}}
 &=\{d\in  {\rm{Der}}(\mathcal{L}):d(\mathcal{I})\subset \mathcal{I}, d(\mathcal{S})=0\}.\endaligned
 $$

\begin{lm} \label{lmZ1}
Let $\mathcal{L}=\mathcal{S}\dot + \mathcal{I}$ be a   semisimple
Leibniz algebra. Then
$$
{\rm{Der}}(\mathcal{L})=R_\mathcal{S}\dot
+{\rm{Der}}(\mathcal{L})_{\mathcal{S}, \mathcal{I}}\dot
+{\rm{Der}}(\mathcal{L})_{\mathcal{I}, \mathcal{I}}.
$$
\end{lm}

\begin{proof}
The proof is similar to that in the case of simple Leibniz
algebras (see \cite[Theorem~3.2]{RMO}).
\end{proof}

\begin{lm} \label{lmZ2}
\begin{itemize}
\item[(a).] For any $d\in {\rm{Der}}(\mathcal{L})_{\mathcal{S}, \mathcal{I}}$,
the restriction $d_{|_\mathcal{S}}:\mathcal{S}\to \mathcal{I}$ is an
$\mathcal{S}$-module homomorphism.
\item[(b).] For any  $d\in {\rm{Der}}(\mathcal{L})_{\mathcal{I}, \mathcal{I}}$,
the restriction $d_{|_\mathcal{I}}:\mathcal{I}\to \mathcal{I}$ is an
$\mathcal{S}$-module homomorphism.
\end{itemize}
\end{lm}

\begin{proof}
(a) Let $d\in {\rm{Der}}(\mathcal{L})_{\mathcal{S}, \mathcal{I}}.$
For each \(x, y\in \mathcal{S}\) we have
\begin{eqnarray*}
d([x, y])=[d(x), y]+[x, d(y)]=[d(x), y].
\end{eqnarray*}
This means that \(d_{|_\mathcal{S}}\) is a \(\mathcal{S}\)-module
homomorphism.

The proof of the part (b) is similar to (a). The proof is
complete.
\end{proof}

The following result directly follows from Lemma~\ref{schur}.

\begin{lm} \label{lmZ2'}
Let $\mathcal{S}$ be a semisimple Lie algebra and $\mathcal{V}$ is a
simple $\mathcal{S}$-module. Let $m,n\in \mathbb{\mathbb{Z}}_+$.
Then
\begin{itemize}
\item[(a).] ${\rm{End}}_\mathcal{S}(m\mathcal{V}, n\mathcal{V})\simeq \hom_{\mathbb{C}}(\mathbb{C}^m, \mathbb{C}^n)$;
\item[(b).]     ${\rm{Aut}}_\mathcal{S}(m\mathcal{V})\simeq {\rm{GL}}_m(\mathbb{C})$.\end{itemize}
\end{lm}

The above four lemmata imply the following result.

\begin{thm} \label{lmZ3}
Let $\mathcal{L}=\mathcal{S}\dot + \mathcal{I}$ be a   semisimple
Leibniz algebra with decompositions given as in~\eqref{decss} and
\eqref{decii}. Then
\begin{itemize}
\item[(a).]  Each $d\in {\rm{Der}}(\mathcal{L})_{\mathcal{S}, \mathcal{I}}$
is uniquely determined by a map in $\oplus_{i=1}^r\hom_{\mathbb{C}}(\mathbb{C}^{n_i}, \mathbb{C}^{m_i})$;
\item[(b).] each $d\in {\rm{Der}}(\mathcal{L})_{\mathcal{I}, \mathcal{I}}$ is uniquely determined by a map in $\oplus_{i=1}^{r+s}\hom_{\mathbb{C}}(\mathbb{C}^{m_i}, \mathbb{C}^{m_i})$;
\item[(c).] \[
\dim {\rm{Der}}(\mathcal{L})=\dim \mathcal{S}+\sum_{i=1}^rn_im_i+\sum\limits_{i=1}^{r+s}
m_i^2.
\]

\end{itemize}
\end{thm}


\section{Automorphisms of semisimple Leibniz algebras}
\label{fifth}

In this section we describe the group Aut$(\mathcal{L})$ of
automorphisms of semisimple Leibniz algebras $\mathcal{L}.$ We
first have the following result.

\begin{lm} \label{lmZ50}
Let $\mathcal{L}=\oplus_{i=1}^nm_i\mathcal{L}_i$ be a semisimple
Leibniz algebra where  $\mathcal{L}_i$'s are pairwise
nonisomorphic indecomposable semisimple Leibniz algebras. Then
$$
{\rm{Aut}}(\mathcal{L})\cong \prod_{i=1}^n ({\mathbb{S}}_{m_i}\ltimes
{\rm{Aut}} (\mathcal{L}_i)^{m_i})
$$ where
${\mathbb{S}}_{m}$ is the symmetric group of permutations. 
\end{lm}

\begin{proof} It suffices to consider an algebra of the form
\(m \mathcal{L}=\mathcal{L}_1\oplus\ldots \oplus\mathcal{L}_m,\)
where \(\mathcal{L}_i=\mathcal{L}\) for all \(i.\) For \(1\leq i,
j \leq m\) denote by \(\pi_{i j}\) an identical isomorphism from
 \(\mathcal{L}_i\) onto
\(\mathcal{L}_j.\) Then
\[
\left\{\pi=\oplus_{i=1}^m \pi_{i \sigma(i)}: \sigma\in
{\mathbb{S}}_{m}\right\}
\]
is a subgroup of \(\textrm{Aut}(m\mathcal{L}),\) which isomorphic
to the group \({\mathbb{S}}_{m}\) and we may identify its to
\({\mathbb{S}}_{m}.\)

Let \(\varphi\in \textrm{Aut}(m\mathcal{L}).\) Since
\(\mathcal{L}_i\) is an indecomposable Leibniz algebra,
\(\varphi(\mathcal{L}_i)\) is also an indecomposable Leibniz
algebra. Thus there exists an index \(\sigma_\varphi(i)\) such
that \(\varphi(\mathcal{L}_i)=\mathcal{L}_{\sigma_\varphi(i)}.\)
It is clear that \(\sigma_\varphi:i\to \sigma_\varphi(i)\) is a
permutation of the set \(\{1, \ldots, m\}\) and
\(\phi_i=\pi_{\sigma_\varphi(i) i}\circ
\varphi_{|_{\mathcal{L}_i}}\) is an automorphism of
\(\mathcal{L}_i.\)  Then \(\varphi=\pi_\varphi \circ\phi,\) where
\(\pi_\varphi=\oplus_{i=1}^m \pi_{i \sigma_\varphi(i)}\) and
\(\phi=\prod_{i=1}^m \phi_i\in {\rm{Aut}} (\mathcal{L})^{m}.\)
Thus $\textrm{Aut}(m\mathcal{L})= \mathbb{S}_{m} {\rm{Aut}}
(\mathcal{L})^{m}$.  It is easy to see that $ {\rm{Aut}}
(\mathcal{L})^{m}$ is a normal subgroup of $\textrm{Aut}(m\mathcal{L})$. So \(\varphi\in \mathbb{S}_{m}\ltimes {\rm{Aut}}
(\mathcal{L})^{m}.\) The proof is complete.
\end{proof}

Now we need to determine $\textrm{Aut}(\mathcal{L})$ for
indecomposable semisimple Leibniz algebras $\mathcal{L}$. From now
on in this section, we assume that $\mathcal{L}=\mathcal{S}\dot{+}
\mathcal{I}$ is an indecomposable semisimple Leibniz algebra with
decompositions as in~\eqref{decss} and \eqref{decii}.  We define
the projections
$$
\aligned \rho_1: \ &\mathcal{L}\to \mathcal{S}, x+y\mapsto x, \forall x\in \mathcal{S}, y\in \mathcal{I};\\
\rho_2: \ &\mathcal{L}\to \mathcal{I}, x+y\mapsto y, \forall x\in
\mathcal{S}, y\in \mathcal{I}.\endaligned
 $$
Let $\varphi$  be an  automorphism of $\mathcal{L}$,
$\varphi_\mathcal{I}=\rho_2\circ \varphi|_\mathcal{I}$,
$\varphi_1=\rho_1\circ \varphi|_\mathcal{S}$ and
$\varphi_2=\rho_2\circ \varphi|_\mathcal{S}$. Clearly, $\varphi$
is determined by $\varphi_\mathcal{I}$, $\varphi_1$ and
$\varphi_2$. It is clear that \(\varphi(\mathcal{I})=
\mathcal{I}.\)

For any $\mathcal{S}$-module $\mathcal{V}$ and any
$\sigma\in$Aut$(\mathcal{S})$, we define the new
$\mathcal{S}$-module $\mathcal{V}^{\sigma}=\mathcal{V}$ with the
action
$$
v\cdot x=[v, x]'=v\sigma(x),\,\,\forall\ v\in \mathcal{V}, x\in
\mathcal{S}.
$$
We know from \cite{Hum}  that $\mathcal{V}\simeq
\mathcal{V}^{\sigma}$ if $\sigma$ is an inner automorphism of
$\mathcal{S}.$ If $\sigma$ is not an inner automorphism of
$\mathcal{S}$, we generally do not have  $\mathcal{V}\simeq
\mathcal{V}^{\sigma}$.

\begin{lm} \label{lmZ51}
 \begin{itemize}\item[(a).] $\varphi_1\in {\rm{Aut}}(\mathcal{\mathcal{S}})$.
\item[(b).]       $\varphi_\mathcal{I}$ is an $S$-module isomorphism from $\mathcal{I}$ onto $\mathcal{I}^{\varphi_1}$.
\item[(c).] $\varphi_2$ is an $\mathcal{S}$-module hommorphism from $\mathcal{S}$ to $\mathcal{I}^{\varphi_1}$.
\end{itemize}
\end{lm}

\begin{proof}
Let $x,y\in \mathcal{S}$, then
\begin{eqnarray*}
\varphi_1([x,y])+\varphi_2([x,y])&=&
\varphi([x,y])=[\varphi(x),\varphi(y)]=\\
&=& [\varphi_1(x)+\varphi_2(x),
\varphi_1(y)+\varphi_2(y)]=\\
&=&[\varphi_1(x),\varphi_1(y)]+[\varphi_2(x), \varphi_1(y)].
\end{eqnarray*}
This implies
\begin{eqnarray*}
\varphi_1([x,y])=[\varphi_1(x),\varphi_1(y)],
\end{eqnarray*}
\begin{eqnarray*}
\varphi_2([x,y])=[\varphi_2(x), \varphi_1(y)].
\end{eqnarray*}

Let $x\in \mathcal{S}, i \in \mathcal{I}$, then
\begin{eqnarray*}
\varphi_I([i, x]) &=& \varphi([i, x])=[\varphi(i),\varphi(x)]=
[\varphi_I(i), \varphi_1(x)].
\end{eqnarray*}
The proof is complete.
\end{proof}

\begin{lm}\label{lmZ52}
\begin{itemize}
\item[(a).] For $\sigma\in
{\rm{Aut}}(\mathcal{S})$ and an $\mathcal{S}$-module isomorphism
$\sigma_\mathcal{I}: \mathcal{I}\to \mathcal{I}^{\sigma}$ there is
$\varphi\in {\rm{Aut}}(\mathcal{L})$ such that $\varphi_1=\sigma$
and $\varphi_\mathcal{I}=\sigma_\mathcal{I}.$
\item[(b).] For $\sigma\in
{\rm{Aut}}(\mathcal{S})$ there is $\varphi\in
{\rm{Aut}}(\mathcal{L})$ such that $\varphi_1=\sigma$ if and only
if $\mathcal{I}\simeq \mathcal{I}^{\sigma}$ as
$\mathcal{S}$-modules.
\end{itemize}
\end{lm}

\begin{proof}
(a) For $\sigma\in {\rm{Aut}}(\mathcal{S})$ and an $S$-module
isomorphism $\sigma_\mathcal{I}: \mathcal{I}\to
\mathcal{I}^{\sigma}$ set
\[
\varphi(x)=\sigma(x)+\sigma_\mathcal{I}(i),\ x+i\in
\mathcal{S}+\mathcal{I}.
\]
Then
\begin{eqnarray*}
\varphi([x+i, y+j]) & = & \varphi([x, y]+[i, y])=\sigma([x,
y])+\sigma_\mathcal{I}([i, y])=\\
&=& [\sigma(x), \sigma(y)]+[\sigma_\mathcal{I}(i), \sigma(y)]= \\
&=& [\sigma(x)+\sigma_\mathcal{I}(i),
\sigma(y)+\sigma_\mathcal{I}(j)]=[\varphi(x+i), \varphi(y+j)].
\end{eqnarray*}
Thus $\varphi$ is an automorphism such that $\varphi_1=\sigma$ and
$\varphi_\mathcal{I}=\sigma_\mathcal{I}.$

(b) Let $\sigma\in {\rm{Aut}}(\mathcal{S}).$ Suppose that there
exists $\varphi\in {\rm{Aut}}(\mathcal{L})$ such that
$\varphi_1=\sigma.$ For every \(x\in \mathcal{S}, i\in
\mathcal{I}\) we have
\begin{eqnarray*}
\varphi_I([i, x])=\varphi([i, x])=[\varphi(i),
\varphi(x)]=[\varphi_I(i), \sigma(x)].
\end{eqnarray*}
This means that $\mathcal{I}\simeq \mathcal{I}^{\sigma}$ as
$\mathcal{S}$-modules. The converse assertion follows from the part
(b). The proof is complete.
\end{proof}

The following example shows the existence of automorphism of \(\mathcal{S}\) which can not be extended to the whole algebra
\(\mathcal{L}.\)

\begin{exam}
Let $\mathfrak{sl}_3$ be the Lie algebra consisting of all
traceless $3\times 3$ complex matrices, and
let $\mathcal{V}(\Lambda_1)=\mathbb{C}^3$ be the standard
$\mathfrak{sl}_3$-module (by left matrix multiplication). Its
highest weight is the first fundamental weight $\Lambda_1$.
Consider the automorphism of  $\mathfrak{sl}_3$:
$$
\sigma:\mathfrak{sl}_3\to \mathfrak{sl}_3,\,\, x\mapsto -x^t,
$$
for all $x\in \mathfrak{sl}_3$,  where $x^t$ is the   transpose of
$x$. This automorphism $\sigma$ cannot be extended to an automorphism of
the whole simple Leibniz algebra $\mathfrak{sl}_3\ltimes V(\Lambda_1)$
since $\mathcal{V}(\Lambda_1)^\sigma\simeq
\mathcal{V}(\Lambda_2)\not\simeq \mathcal{V}(\Lambda_1),$ where
\(\Lambda_2=-\sigma\circ \Lambda_1\) (see \cite[P. 116, Exersice
21.6]{Hum}).
\end{exam}

We define the following subgroups of Aut$(L)$:
$$\aligned
{\rm{Aut}}(\mathcal{L})_\mathcal{I}=&\{\varphi\in
{\rm{Aut}}(\mathcal{L}): \varphi|_\mathcal{S}=
{\rm{id}}_\mathcal{S}\},\\
{\rm{Aut}}(\mathcal{L})_\mathcal{S}=&\{\varphi\in {\rm{Aut}}(\mathcal{L}):  \varphi(\mathcal{S})=\mathcal{S}\},\\
{\rm{Aut}}(\mathcal{L})_0=&\{\varphi\in {\rm{Aut}}(\mathcal{L}):
\varphi|_\mathcal{I}=
{\rm{id}}_\mathcal{I}\}.\\
\endaligned
$$

\begin{lm}\label{groupauto1}
Let $\mathcal{L}=\mathcal{S}\dot{+} \mathcal{I}$ be a semisimple
Leibniz algebra.  Then we have the following isomorphisms:
$$
{\rm{Aut}}(\mathcal{L})_\mathcal{I} \cong
{\rm{Aut}}_S(\mathcal{I}),
 \quad {\rm{Aut}}(\mathcal{L})_0 \cong {\rm{hom}}_S(\mathcal{S},\mathcal{I}).$$
\end{lm}

\begin{proof} Let us prove the first isomorphism. For an arbitrary \(\varphi\in \textrm{Aut}(\mathcal{L})_\mathcal{I}\) and $x\in \mathcal{S}, \ i \in \mathcal{I}$ we have
\begin{eqnarray*}
\varphi([i, x]) &=& [\varphi(i), \varphi(x)]= [\varphi(i), x].
\end{eqnarray*}
Thus \(\varphi_{|_\mathcal{I}}\in  {\textrm{Aut}}_S(\mathcal{I}).\)

Let now \(\sigma\in  {\textrm{Aut}}_S(\mathcal{I}).\) Set
\begin{eqnarray}\label{oneiso}
\varphi_\sigma(x+i)=x+\sigma(i),\, x+i\in \mathcal{S}+\mathcal{I}.
\end{eqnarray}
Then \(\varphi_\sigma\in \textrm{Aut}(\mathcal{L})_\mathcal{I},\)
and  \eqref{oneiso} implies that the mapping \(\sigma \to
\varphi_\sigma\) is an isomorphism from ${\rm{Aut}}_S(\mathcal{I})$ to ${\rm{Aut}}(\mathcal{L})_\mathcal{I}$.

Now we shall prove the second assertion. Let \(\varphi\in \textrm{Aut}(\mathcal{L})_0.\)  Since
\begin{eqnarray*} [i,
x]=\varphi([i,x])=[\varphi(i), \varphi(x)]=[i, \varphi_1(x)],
\end{eqnarray*}
we obtain \([i, \varphi_1(x)-x]=0\) for all \(i\in
\mathcal{I}, x\in \mathcal{S}.\) Thus \([\mathcal{I},
\mathcal{S}_{\varphi_1(x)-x}]=0,\) where
\(\mathcal{S}_{\varphi_1(x)-x}\) is the ideal of \(\mathcal{S}\)
generated by the element \(\varphi_1(x)-x.\) Thus
\(\varphi_1(x)=x\) for all \(x\in \mathcal{S},\) i.e.,
\(\varphi_1=\textrm{id}_\mathcal{S}.\) So, we have
\begin{eqnarray}\label{sss}
\varphi\in \textrm{Aut}(\mathcal{L})_0\, \Rightarrow \,
\varphi_1=\textrm{id}_\mathcal{S}.
\end{eqnarray}

The equalities
\begin{eqnarray*}
\varphi([x, y])=[x, y]+\varphi_2([x, y]),
\end{eqnarray*}
\begin{eqnarray*}
[\varphi(x), \varphi(y)]=[x+\varphi_2(x), y+\varphi_2(y)]=[x, y]+
[\varphi_2(x), y], \forall x, y\in \ \mathcal{S}
\end{eqnarray*}
imply that \(\varphi_2([x, y])=[\varphi_2(x), y]\). Therefore,
\(\varphi_2\in \textrm{hom}_S(\mathcal{S}, \mathcal{I}).\)

Let \(\sigma\in  \textrm{hom}_S(\mathcal{S}, \mathcal{I}).\) Set
\begin{eqnarray}\label{secondiso}
\varphi_\sigma(x+i)=x+\sigma(x)+i,\, x+i\in
\mathcal{S}+\mathcal{I}.
\end{eqnarray}
Then \(\varphi_\sigma\in \textrm{Aut}(\mathcal{L})_0,\) and by
\eqref{secondiso} we see that the mapping \(\sigma \to
\varphi_\sigma\) is an isomorphism of the additive group
$\textrm{hom}_S(\mathcal{S}, \mathcal{I})$ and the multiplicative group ${\textrm{Aut}}(\mathcal{L})_0.$
\end{proof}

Denote
\[
\textrm{Aut}(\mathcal{S})_0=\{\varphi\in {\rm{Aut}}(\mathcal{S}):
\exists \ \psi\in {\rm{Aut}}(\mathcal{L})_\mathcal{S},
\psi|_\mathcal{S}=\varphi\}.
\]
Lemmata~5.2, 5.3, \ref{groupauto1} and implication \eqref{sss} gives us the following matrix representations of the above mentioned groups:
$$\aligned
&{\rm{Aut}}(\mathcal{L})= \left\{\left(%
\begin{array}{cc}
  \varphi_1 & 0 \\
  \varphi_2 & \varphi_\mathcal{I}\\
\end{array}%
\right): \varphi_1 \in {\rm{Aut}}(\mathcal{S})_0,\, \varphi_\mathcal{I}
 \in {\rm{Iso}}_{\mathcal{S}}(\mathcal{I}, \mathcal{I}^{\varphi_1}), \ \varphi_2\in \textrm{hom}_{\mathcal{S}}(\mathcal{S}, \mathcal{I}^{\varphi_1})\right\},\\
&{\rm{Aut}}(\mathcal{L})_\mathcal{I}= \left\{\left(%
\begin{array}{cc}
  \textrm{id}_\mathcal{S} & 0 \\
  0 & \varphi_\mathcal{I}\\
\end{array}%
\right): \varphi_\mathcal{I} \in {\rm{Aut}}_{\mathcal{S}}(\mathcal{I})\right\},\\
&{\rm{Aut}}(\mathcal{L})_\mathcal{S}=  \left\{\left(%
\begin{array}{cc}
  \varphi_1 & 0 \\
  0 & \varphi_\mathcal{I}\\
\end{array}%
\right): \varphi_1 \in {\rm{Aut}}(\mathcal{S})_0,\, \varphi_\mathcal{I} \in{\rm{Iso}}_{\mathcal{S}}(\mathcal{I}, \mathcal{I}^{\varphi_1})\right\},\\
&{\rm{Aut}}(\mathcal{L})_0= \left\{\left(%
\begin{array}{cc}
  \textrm{id}_\mathcal{S} & 0 \\
  \varphi_2 & \textrm{id}_\mathcal{I} \\
\end{array}%
\right): \varphi_2\in \textrm{hom}_{\mathcal{S}}(\mathcal{S}, \mathcal{I})\right\},\\
\endaligned
$$
where ${\rm{Iso}}_{\mathcal{S}}(\mathcal{I}, \mathcal{I}^{\varphi_1})$ is the set of $\mathcal{S}$-module isomorphisms.

These representations show that
\({\rm{Aut}}(\mathcal{L})_\mathcal{I}\) and
\({\rm{Aut}}(\mathcal{L})_0\) are normal subgroups of
\({\rm{Aut}}(\mathcal{L})_\mathcal{S}\) and
\({\rm{Aut}}(\mathcal{L}),\) respectively.

\begin{cor}\label{groupauto2}
Let $\mathcal{L}=\mathcal{S}\dot{+} \mathcal{I}$ be a semisimple
Leibniz algebra.    \begin{itemize}
\item[(a).] $
{\rm{Aut}}(\mathcal{S})_0 \cong {\rm{Aut}}(\mathcal{L})_\mathcal{S}/
{\rm{Aut}}(\mathcal{L})_\mathcal{I}.
$
\item[(b).] If all $\mathcal{S}_i$ are not of type $A_l, D_l$ or $E_6$, then ${\rm{Aut}}(\mathcal{S})_0 ={\rm{Aut}}(\mathcal{S})$.
\end{itemize}
 \end{cor}

Now we can prove the main result of this section.

\begin{thm}\label{groupauto2}
Let $\mathcal{L}=\mathcal{S}\dot{+} \mathcal{I}$ be a
semisimple Leibniz algebra.  Then
$$
{\rm{Aut}}(\mathcal{L})={\rm{Aut}}(\mathcal{L})_S   \ltimes
  {\rm{Aut}}(\mathcal{L})_0.
$$
 \end{thm}

\begin{proof} Let \(\varphi\in  {\rm{Aut}}(\mathcal{L}).\)
Set
\[
\psi(x+i)=\varphi_1(x)+\varphi_I(i),\ x+i \in
\mathcal{S}+\mathcal{I}.
\]
Let us show that \(\psi\) is also an automorphism. Indeed,
\begin{eqnarray*}
\psi([x+i,
y+j])&=&\psi([x,y]+[i,y])=\varphi_1([x,y])+\varphi_I([i,y])\\
&=& [\varphi_1 (x), \varphi_1 (y)]+\varphi([i,y])= [\varphi_1(x),
\varphi_1(y)]+[\varphi(i),
\varphi(y)]\\
&=& [\varphi_1(x),\varphi_1(y)]+[\varphi_I(i),
\varphi_1(y)+\varphi_I(y)]\\
&=& [\varphi_1(x)+\varphi_\mathcal{I}(i),\varphi_1(y)+
\varphi_{\mathcal{I}}(j)]=[\psi(x+i),\psi(y+j)].
\end{eqnarray*}
It is clear that \(\psi\in {\rm{Aut}}(\mathcal{L})_\mathcal{S}.\)

 Now consider the automorphism
\[
\eta=\psi^{-1}\circ\varphi.
\] Then
\[
\eta(x+i)=x+\varphi_{\mathcal{I}}^{-1}\varphi_2(x)+i
\]
and  therefore  \(\eta=\psi^{-1}\circ\varphi\in
{\rm{Aut}}(\mathcal{L})_0.\) Since \(\varphi=\psi\circ\eta ,\) it
follows that
$$
 {\rm{Aut}}(\mathcal{L})={\rm{Aut}}(\mathcal{L})_S   \ltimes
  {\rm{Aut}}(\mathcal{L})_0.$$
The proof is complete.
\end{proof}

So far we have described the automorphism groups for all semisimple Leibniz algebras.

\

\begin{center}
\bf Acknowledgments
\end{center}

\noindent K.Z. is partially supported by  NSF of China (Grant
11271109) and NSERC.

\end{document}